\documentclass[a4paper,12pt]{article}

\usepackage{amsmath, amssymb, amsthm}

\newtheorem{thm}{Theorem}
\newtheorem{cor}[thm]{Corollary}

\newtheorem{prop}[thm]{Proposition}

\begin{document}

\title{Notes on Loewy series of centers of 
$p$-blocks}
\author{Yoshihiro Otokita\footnote{Department of Mathematics and Informatics, Chiba University, Chiba--shi 263--8522, Japan. E--mail: otokita@chiba-u.jp}}
\date{}

\maketitle

\begin{abstract}
 The present paper describes some results on the Loewy series of the center of a modular group algebra in order to solve a problem on the number of irreducible ordinary characters. 
 For instance, we prove that a $p$-block of a finite group has at least $p+2$ characters if its defect group contains non-elementary center. 
\end{abstract}


\section{Introduction}
 In this paper we consider a problem on the number of irreducible ordinary characters of a finite group through some studies of the center of a modular group algebra.

Let $G$ be a finite group and $F$ an algebraically closed field of characteristic $p>0$. For a $p$-block $B$ of the group algebra $FG$ with defect $d$, we denote by $k(B)$ the number of associated irreducible ordinary characters. In \cite{HK} H\'{e}thelyi and K\"{u}lshammer conjectured that:
\[\textit{$2\sqrt{p-1} \le k(B)$ unless $d=0$?} \]
This is a refinement of a result by Mar\'{o}ti \cite{M} (see also \cite{MM}) and is still unsolved in general. In modular representation theory it is known that $k(B)$ is equal to the dimension of the center $ZB$ of $B$. We accordingly examine the algebraic structure of $ZB$ in order to solve the problem above. 

In the next section we prove the following results.

We first determine the Loewy structure of $ZB$ for cyclic defect groups. More precisely, it is shown that each codimension of the Loewy series of $ZB$ is equal to $1$ or the inertial index of $B$. Moreover we refer to a remark on 
$ZB$ related to the problem above. 

The second subsection deals with bounds for the Loewy length $LL(ZB)$, that is, the nilpotency index of the Jacobson radical of $ZB$. Some papers \cite{KOS, KS, Ok, Ot} have discussed its upper bounds: for example, Okuyama has stated $LL(ZB) \le p^{d}$. The main purpose of this section is to give a lower bound for $LL(ZB)$. As an immediate corollary to this bound, we conclude that $B$ has at least $p+2$ characters provided its defect group contains non-elementary center.

Throughout this paper we use the following notations.

We fix a defect group $D$ of order $p^{d}$. Let $N_{G}(D, b)$ be the inertial subgroup of a root $b$ of $B$ in $N_{G}(D)$. Then the inertial index of $B$ is $e(B) = |N_{G}(D, b) : DC_{G}(D) |$. As usual $l(B)$ denotes the number of irreducible Brauer characters associated to $B$. For a finite-dimensional algebra $A$ over $F$ with Jacobson radical $J(A)$, its Loewy length $LL(A)$ is defined to be the smallest positive integer $t$ such that $J(A)^{t} = 0$. For an integer $n \ge 1$ we define $c(n)$ as the codimension of $J(A)^{n}$ in $J(A)^{n-1}$. Then $c(n) \neq 0$ for any $1 \le n \le LL(A)$.

\section{Results}
\subsection{Cyclic defect groups}

 In this section we study $c(n)$ for $ZB$. It always holds $c(1)=1$ as $ZB$ is local. Moreover it is clear that $c(2)=0$ if and only if $d=0$. We next see the case of $c(2)=1$ in the proposition below. This is a corollary to {\cite[Proposition 8]{KOS}} since it is easy to check $c(n) \le 1$ for all $3 \le n$. 

\begin{prop} \label{Prop1} If $d > 0$, then the following are equivalent:
\begin{enumerate}
\item $c(2)=1$;
\item $ZB$ is uniserial;
\item $B$ is nilpotent and $D$ is cyclic.
\end{enumerate}
\end{prop}

In the next theorem we determine all the codimensions $c(n)$ for cyclic defect groups. In the proof below we use the algebra of fixed points associated to $B$. As $N_{G}(D, b)$ acts on $Z(D)$ we  define
\[ F[Z(D)]^{N_{G}(D, b)} = \{ a \in F[Z(D)] \ ; \ \text{$t^{-1}at = a$ for $t \in N_{G}(D, b)$} \}. \]
The last part of this theorem is due to {\cite[Corollary 2.8]{KKS}}.

\begin{thm} \label{Thm2} If $D$ is cyclic, then 
\begin{equation*}
c(n) = \begin{cases}
               1  & (n=1) \\
               e(B) & (n=2) \\
               1 &  (3 \le n \le LL(ZB)) 
             \end{cases}
 \end{equation*}
and
\[ LL(ZB) = \frac{p^{d} - 1}{e(B)} + 1. \]
\end{thm}
                
\begin{proof}
Put $I = N_{G}(D, b) / C_{G}(D)$. Since $B$ is perfectly isometric to its Brauer correspondent in $N_{G}(D)$ we may assume $B = FH$ where $H = D \rtimes I$ by \cite{K}. Let $\Gamma$ be the ideal of $ZFH$ spanned by all class sums of defect $0$. Then $ZFH = FD^{N_{G}(D, b)} \oplus \Gamma$ (cf. {\cite[Theorem 1.1]{BS}}). Since $\Gamma$ is contained in the socle of $FH$ we have $\Gamma^{2}=0$. Thus $J(ZFH) = J(FD^{N_{G}(D, b)}) \oplus \Gamma$ and $J(ZFH)^{2} = J(FD^{N_{G}(D, b)})^{2}$.  As $FD^{N_{G}(D, b)}$ is uniserial (see the proof of {\cite[Corollary 2.8]{KKS}}), it follows that $c(n)=1$ for $3 \le n \le LL(ZB)$. Therefore $c(2)=k(B) - 1 - \{ LL(ZB) - 2\} = e(B)$ by Dade's theory.
\end{proof}

For cyclic defect groups $l(B)=e(B)$ and
\[ k(B) = \frac{p^{d}-1}{e(B)} + e(B) \ge 2 \sqrt{p^{d}-1}. \]
In general the relations between $c(n), l(B)$ and $e(B)$ are as yet unknown. However it is possible to answer the question in Introduction under the conditions below inspired by Theorem \ref{Thm2}.

We now assume $d > 0$ and one of the following holds:

\begin{itemize}
\item[(i)] $e(B) \le c(n)$ for some $2 \le n \le LL(ZB)$;
\item[(ii)] $e(B) \le l(B)$.
\end{itemize}

Then we obtain $2 \sqrt{\lambda - 1} \le k(B)$ where $\lambda = LL(F[Z(D)])$ as follows:

By {\cite[Corollary 2.7]{KKS}},
\[ \frac{\lambda - 1}{e(B)} + 1 \le LL(F[Z(D)]^{N_{G}(D, b)}) \le LL(ZB). \]
It is known that $l(B)$ is equal to the dimension of the Reynolds ideal of $ZB$. Since it is contained in the socle of $ZB$, 
\[ k(B) \ge 1 + e(B) + \{(\lambda - 1)/e(B) - 1\} \ge 2 \sqrt{\lambda - 1} \]
in either case.


If $Z(D)$ has type $(p^{a_{1}}, \dots, p^{a_{r}})$, then $\lambda = p^{a_{1}} + \dots + p^{a_{r}} - r + 1$. Therefore we deduce $2 \sqrt{p-1} \le k(B)$ in the assumptions.


\subsection{Lower bounds}

In this subsection we consider a lower bound for $LL(ZB)$. 
By a result of Brou\'{e} {\cite[Proposition (III) 1.1]{B}} there exists an ideal $K$ of $ZB$ such that $ZB/K$ is isomorphic to $F[Z(D)]^{N_{G}(D, b)}$. Thereby $LL(F[Z(D)]^{N_{G}(D, b)}) \le LL(ZB)$. In the next theorem we give a different lower bound for the left side from {\cite[Corollary 2.7]{KKS}}. 

 \begin{thm} \label{Thm4} Let $p^{m}$ be the exponent of $Z(D)$. Then
 \[ \frac{p^{m}+p-2}{p-1} \le LL(F[Z(D)]^{N_{G}(D, b)}) \le LL(ZB). \]
 \end{thm}
 
 \begin{proof}
 We fix an orbit $\mathcal{O}$ 
 of an element in $Z(D)$ with maximal order by the action of $N_{G}(D, b)$. 
 Remark that $|\mathcal{O}|$ divides $e(B)$ since $DC_{G}(D)$ acts trivially on $Z(D)$, namely $|\mathcal{O}| \neq 0$ in $F$. We put $a = |\mathcal{O}| 1 - \sum_{u \in \mathcal{O}} u$ where $1$ is the unit in $G$. As $a$ is contained in $J(F[Z(D)]^{N_{G}(D, b)})$ it suffices to prove that $a^{t} \neq 0$ where $t = 1 + p + \dots + p^{m-1}$. For $0 \le i \le m-1$, 
 \[ a^{p^{i}} = |\mathcal{O}|^{p^{i}} 1 - \sum_{u \in \mathcal{O}} u^{p^{i}} = |\mathcal{O}| 1 - \sum_{u \in \mathcal{O}} u^{p^{i}}\]
  by Fermat's theorem. Hence each term of $a^{t} = a \cdot a^{p} \cdots a^{p^{m-1}}$ has the form
 \[ (-1)^{|I|} |\mathcal{O}|^{m- |I|} \prod_{i \in I} (u_{i})^{p^{i}} \]
 where $I \subseteq \{ 0, 1, \dots, m-1\}$ and $u_{i} \in \mathcal{O}$. Suppose now that $\prod_{i \in I} (u_{i})^{p^{i}} = 1$ for some $I \neq \emptyset$. Since the order of $u_{i}$ is $p^{m}$ for any $i \in I$, we may assume $|I| \ge 2$. If $r = \min \{ I \}$ and $s = \min \{ I - \{r\} \}$, we obtain
 \[1 = (1)^{p^{m-s}} = \{ \prod_{i \in I} (u_{i})^{p^{i}} \}^{p^{m-s}} = \prod_{i \in I}(u_{i})^{p^{m-s+i}} = (u_{r})^{p^{m-s+r}} \neq 1, \]
 a contradiction. Thus the coefficient of $1$ in $a^{t}$ is $|\mathcal{O}|^{m} \neq 0$. Therefore $a^{t} \neq 0$ as claimed.
 \end{proof}
 
 In Theorem \ref{Thm4} it seems that $Z(D)$ is generally not replaced with $D$ itself. Suppose that $B$ is a $p$-block with defect group
 \[ M_{p^{d}} = < x, y \ ; \ x^{p^{d-1}} = y^{p} = 1, y^{-1}xy = x^{1+p^{d-2}} > \]
 where $d \ge 4$. As is well known, $M_{p^{d}}$ has exponent $p^{d-1}$ and has cyclic center of order $p^{d-2}$.  By {\cite[Proposition 10]{KS}},
 \[ LL(ZB) = \frac{p^{d-2}-1}{l(B)} + 1 \le p^{d-2} < \frac{p^{d-1}+p-2}{p-1}. \]
 Furthermore we note that Theorem \ref{Thm4} is clear in this case since $l(B) \le p-1$ (cf. {\cite[Theorem 8.1, 8.8]{S}}). 
 
As a corollary to this bound and {\cite[Proposition 2.2]{Ot}}, we deduce:

\begin{cor} \label{Cor5} Let $p^{m}$ be the exponent of $Z(D)$. Then
\[ \frac{p^{m}+p-2}{p-1} \le k(B) - l(B) + 1 \le k(B). \]
 In particular, $p+2 \le k(B)$ provided $m \ge 2$. 
\end{cor}

At the end of this section we reconsider the previous remark. 

Here assume $d > 0$ and one of the following holds:
\begin{itemize}
\item[(iii)] $Z(D)$ is not elementary abelian;
\item[(iv)] $e(B) \le mr \cdot c(n)$ for some $2 \le n \le LL(ZB)$;
\item[(v)] $e(B) \le mr \cdot l(B)$
\end{itemize}
where $p^{m}$ and $r$ are the exponent and the rank of $Z(D)$, respectively. Then $2 \sqrt{p-1} \le k(B)$ by Corollary \ref{Cor5} and the methods in (i) and (ii).



\section{Examples}

We introduce two examples by Brough-Schwabrow \cite{BS1} and H\'{e}thelyi-K\"{u}lshammer \cite{HK} in order to understand our theorems. In the following let $P$ be a Sylow $p$-subgroup of $G$. 

\begin{enumerate}
\item Assume $p=3$ and $G={}^{2}G_{2}(q)$ is the simple Ree group where $q$ is a power of $3$ at least $27$. 
Then $FG$ decomposes into two $3$-blocks: the principal block $B_{0}$ with $q+7$ characters and simple block $B_{1}$ with the Steinberg character. 
By {\cite[Theorem 1.1]{BS1}}, we have $LL(ZB_{0}) = 3$.
In this case Theorem \ref{Thm4} indicates that $Z(P)$ has exponent $3$. In fact 
\[ Z(P) = \{ x(0, 0, v) \ ; \ v \in {\mathbb F}_{q} \} \]
 by the notations {\cite[page 59]{BS1}} and 
 \[ x(0, 0, v)^{3}=x(0, 0, 3v) = x(0, 0, 0). \] 

\item We denote by $k(G)$ and $l(G)$ the numbers of irreducible ordinary and Brauer characters of $G$, respectively and apply Corollary \ref{Cor5} to $P$ as a defect group of the principal block. Then it follows that $p+2 \le k(G)$ when $Z(P)$ is non-elementary, otherwise it is likely that $k(G)-l(G)+1$ is "small". Now assume $p \equiv 23 \pmod{264}$. Let $X$ be a cyclic group of order $x = (p-1)/22$ and $H$ the unique $2$-fold covering group of the four-degree symmetric group. Then $H \times X$ acts freely on an elementary abelian group $P$ of order $p^{2}$ and the Frobenius group $G = P \rtimes (H \times X)$ satisfies
\[ k(G) = \frac{p^{2}-1}{48x} + 8x = \frac{217}{264}p + \frac{25}{264} \]
as mentioned in {\cite[Remark (iii)]{HK}}. Since $P$ is a normal Sylow $p$-subgroup of $G$, 
\[ l(G) = k(H \times X) = 8x = \frac{4}{11}p - \frac{4}{11} \]
and
\[ k(G) - l(G) + 1 = \frac{11}{24}p + \frac{35}{24}. \]

\end{enumerate}

\end{document}